\documentclass{amsart}

\usepackage{amsmath}
\usepackage{amssymb}
\usepackage{amsfonts}
\usepackage{enumerate}
\usepackage{vmargin}

\newcommand{\bC}{{\mathbb{C}}}

\newcommand{\bR}{{\mathbb{R}}}

  \newcommand{\M}{{\mathcal{M}}}

  \newcommand{\R}{{\mathcal{R}}}

\newcommand{\ep}{\varepsilon}
\renewcommand{\phi}{\varphi}
\newcommand{\upchi}{{\raise.35ex\hbox{\ensuremath{\chi}}}}
\newcommand{\eps}{\varepsilon}

\renewcommand{\leq}{\leqslant}
\renewcommand{\geq}{\geqslant}

\newtheorem{thm}{Theorem}[section]

\newtheorem{cor}[thm]{Corollary}
\newtheorem{lemma}[thm]{Lemma}
\newtheorem{rk}[thm]{Remark}

\begin{document}

\title{An $L_p$-inequality for anticommutators}
\date{}
\author[\'E. Ricard]{\'Eric Ricard}
\address{Normandie Univ, UNICAEN, CNRS, Laboratoire de Math{\'e}matiques Nicolas Oresme, 14000 Caen, France}
\email{eric.ricard@unicaen.fr}


\begin{abstract}
We prove a basic inequality involving anticommutators in noncommutative
$L_p$-spaces. We use it to complete our study of the noncommutative Mazur
maps from $L_p$ to $L_q$ showing that they are Lipschitz on balls when $0<q<p<\infty$. 
\end{abstract}
\maketitle
\section{Introduction}

This short note deals with noncommutative integration in von Neumann algebras.
We refer to \cite{PX} for basic definitions and notation  in the semifinite or type III setting. We will simply denote by $L_p$ the noncommutative  $L_p$-space associated
to a semifinite von Neumann algebra $\M$ with a trace
$\tau$ or a general von Neumann algebra $\M$ with a weight $\phi$.

The Mazur maps $M_{p,q}$ are bijections between $L_p$ and $L_q$
defined by $M_{p,q}(f)=f|f|^{p/q-1}$ for $0<p,q<\infty$. They are
known to be uniformly continuous on balls by \cite{Ray}.  Very few
precise estimates were known. Kosaki gave some of them \cite{Kos,Kos2}
and the work of Alexandrov and Peller \cite{AP} on the functional
calculus can be used for the case $1<q<p$ on Schatten classes. In
several papers, we tried to quantify this continuity.  In \cite{R}, we
treated the Banach space case $1\leq p,q<\infty$. Just as in the
commutative case, when $p>q$, $M_{p,q}$ is globally $\frac
pq$-H\"older and simply Lipschitz on balls when $q<p$. The involved
tools make a crucial use of convexity that is not available below
exponent 1. For instance, unital completely positive Schur multipliers
are no longer bounded on the Schatten $p$-class $S_p$ when $p<1$. They
are used in lot of inequalities in noncommutative $L_p$ and are
closely related to functional calculus; this explains why the
quasi-Banach situation is much more difficult to handle and pretty
much unexplored.  The main inequality of \cite{PR} used
plurisubharmonicity of the $L_p$-norm (in the finite case) to overcome
this lack of convexity.  It leads to non optimal estimates for the
modulus of continuity of $M_{p,q}$ is the quasi-Banach case. To deal
with $p<q$ in \cite{R2}, we used precise estimates on the boundedness
of Schur multipliers on $S_p$ as well as a suitable decomposition of
operators to get the optimal rate. The main motivation for this
paper is to settle the remaining case $q<p$ when $q<1$.

Our main argument is a variation on \cite{PR} combined with general
techniques from \cite{R}. Our main inequality is written in terms of
an anticommutator estimate which barely says that a kind of weighted
triangular truncation is bounded from $L_p$ to $L_q$. Not so
surprisingly, it implies the $L_p$-$L_q$ boundedness of some Schur
type multipliers that turn out to be enough to deal with Mazur maps.
The proofs of the inequalities are made for semifinite von Neumann algebras
(finite is enough by \cite{R}). The general case can
be obtained using the Haagerup reduction principle as in \cite{R}.

We will focus only on the relevant inequalities and skip
technical approximation arguments that can be found in details in \cite{PR, R, R2}. As often, we write $C_{p,q}$ for constants that depends on the parameter $p$, $q$ and may differ from line to line.

\section{Inequalities}
\subsection{Main result}

Let $(\M,\tau)$ be a semifinite von Neumann algebra.
\begin{thm}\label{main}
  Fix $\alpha>0$, $0<s<\infty$ and $0<r\leq \infty$. Let $p$ be so
  that $\frac 1 p = \frac 1s+\frac 1r$ and $q$ so that $\frac 1 q = \frac {1+\alpha}s+\frac 1r$. Then there is a constant $C_{\alpha,q}$ so that for any $d\in L_s^+$  and $x\in L_r$:
$$\| xd^{1+\alpha}\|_q \leq C_{\alpha,q} \|d\|_s^\alpha .\| dx+xd\|_p.$$
\end{thm}

The proof will be a variation on the arguments of \cite{PR}. We will
need a bit of notation. Let $\Delta$ be the unit strip in $\bC$ with boundaries $\partial_k$ for $k=0,1$:
$$\Delta= \{ z\in \bC \;;\; 0<{\rm Re}\, z<1\}; \quad \partial_i= \{
z\in \bC \;;\; {\rm Re}\, z=k\}.$$ For $0<\gamma<1$, we will denote by
$\mathbb P^\gamma$ the Poisson measure on the boundary of $\Delta$
relative to the point $\gamma$. It is explicitly given by densities
on $\partial_k$ with respect to the Lebesgue measure, for $t\in \bR$:
$$ Q_\gamma(k+it)= \frac{\sin(\gamma\pi)}{2\big(\cosh(\pi t)- (-1)^k \cos(\gamma \pi)\big)}.$$

For a  $A\subset \partial_0\cup \partial_1$ , we write $2.A$ for its dilation by a factor 2 in the imaginary direction.

\begin{lemma}\label{meseq}
  There is a constant $C_\gamma>0$, such that for any borel set
  $A\subset \partial_0\cup \partial_1$, one has $\mathbb
  P^\gamma(2.A)\leq C_\gamma \mathbb P^\gamma(A)$.
\end{lemma}

\begin{proof}
We need the obvious fact that $\mathbb P^\gamma$ is equivalent to $\mu$
the measure on $\partial_0\cup \partial_1$ with density $Q(k+it)=
\frac{1}{\cosh(\pi t)}$ with a constant that depends only on $\gamma$:
$$ k=0,1,\quad \forall t\in \bR,\qquad Q_\gamma(k+it)\leq \frac {\sin(\gamma \pi)}{2\big(1-|\cos(\gamma\pi)|\big)}Q(k+it),
\quad  Q(k+it)\leq \frac 4{\sin(\gamma \pi)} Q_\gamma(k+it).$$
Thus it suffices to prove the lemma for $\mu$. But for any borel set $B$ in $\mathbb R$ 
$$ \int_{2B} \frac {dt}{\cosh(\pi t)} = 2 \int_B \frac {du}{\cosh(2\pi u )}
\leq 2 \int_B \frac {dt}{\cosh(\pi t )}.$$
It follows that $\mu(2.A)\leq 2\mu(A)$ for any borel set of $\partial_0\cup \partial_1$.

  \end{proof}
Let us recall the complex uniform convexity of $L_q$ from \cite{Xust},
see also \cite[Remark 2.8]{PR}.  To avoid technical discussions, we
will consider only the set $AF(\Delta)$ of holomorphic functions
$F:\Delta\to \M$ that are bounded and extend continuously on
$\partial_k$, $k=0,1$ with values in a given finite von Neumann
subalgebra of $\M$ (depending on $F$). For such an $F$ and
$0<\gamma<1$, we denote by $\|F\|_{L_q^\gamma}$ its $L_q$ norm in with
respect to the measure $\mathbb P^\gamma$ considering it with values in $L_q(\M,\tau)$.

\begin{thm}\label{conv}
Let $0<q\leq 2$, there is a constant $\delta_q>0$ such that for any
$F\in AF(\M)$ and $0<\gamma<1$:
$$\|F(\gamma)\|_q^2 + \delta_q \|F-F(\gamma)\|_{L_q^\gamma}^2 \leq \|F\|_{L_q^\gamma}^2.$$
  \end{thm}

\begin{proof}[Proof of Theorem \ref{main}]
  By standard approximation arguments (see \cite[Proof of 2.1]{PR}), we
  can assume that $d$ and $x$ sit in a finite von Neumann algebra and
  that $d$ is invertible with $\|d\|_s=1$ and has a finite discrete spectrum. Moreover using a $2\times2$ matrix trick, we can assume
  $x=x^*$ loosing on the constant $C_{\alpha,q}$. This is to ensure that
  $\|xd^{1+\alpha}\|_q=\|d^{1+\alpha}x\|_q$.

  We start by the most difficult case $q\leq 1$.

  Let $1>\eta>0$ to be fixed latter and set $\gamma=\frac {\alpha}{1+\alpha}\in (0,1)$. We first assume that 
  $\| xd+dx\|_p\leq \eta \|x d^{1+\alpha}\|_q$.

  By the $q$-triangle and the H\"older inequalities, we have
  $$\| x d^{1+\alpha} \|_q^q=\|  d^{1+\alpha}x \|_q^q \leq \|d^\alpha x d\|_q^q + \|d^\alpha(xd+dx)\|_q^q\leq
  \|d^\alpha x d\|_q^q+ \eta^q \| x d^{1+\alpha} \|_q^q.$$
  Thus $(1-\eta^q) \| x d^{1+\alpha} \|_q^q \leq \|d^\alpha x d\|_q^q$.

  Because of all the assumptions we made on $d$, the function given by
  $F(z)= d^{(1+\alpha)z} xd^{(1+\alpha)(1-z)}$ is in
  $AF(\M)$. We have $F(\gamma)=d^\alpha xd$ and moreover for all $t\in\bR$ as $x=x^*$
  $$\|F(it)\|_q=\|xd^{1+\alpha}\|_q=\|d^{1+\alpha}x\|_q=\|F(1+it)\|_q,$$
  thus $\|F\|_{L_q^\gamma}=  \|xd^{1+\alpha}\|_q$. By Theorem \ref{conv},
  $$\delta_q \| F- d^\alpha xd\|_{L_q^\gamma}^2\leq \|xd^{1+\alpha}\|_q^2-
  \|d^\alpha xd\|_q^2\leq \big(1 - (1-\eta^q)^{2/q}\big)\|xd^{1+\alpha}\|_q^2 .$$

  Let  $\ep= 2.6^{-1/q}$ and $A=\{ z\in \partial_0\cup \partial_1\; |\;
  \| F(z)-d^\alpha xd\|_q \geq \ep \|xd^{1+\alpha}\|_q\}.$

  By the Markov inequality,
  $$\mathbb P^\gamma (A)\leq \frac{ \big(1 -
    (1-\eta^q)^{2/q}\big)^{q/2}}{\delta_q^{q/2}\ep^q}.$$ By Lemma
  \ref{meseq}, if $\eta$ is chosen small enough, we have that $\mathbb
  P^\gamma(A\cup 2.A)\leq \gamma/2$. Since $\mathbb
  P^\gamma(\delta_1)=\gamma$, we can find $z\in \delta_1\setminus
  (A\cup 2 A).$ Let $z=1+2it$ with $t\in\bR$ and put
  $u_h=d^{(1+\alpha)ih}$ which is obviously an unitary commuting with $d$
  for all $h\in\bR$. By the definition of $2.A$ and $A$ we have
$$ \| d^{1+\alpha}u_t xu_{-t}-d^\alpha xd\|_q, \| d^{1+\alpha}u_{2t} xu_{-2t}-d^\alpha xd\|_q \leq \ep\|xd^{1+\alpha}\|_q.$$ 
  By unitary invariance, the second inequality is $\| d^{1+\alpha}u_{t} xu_{-t}-u_{-t}d^\alpha xdu_t\|_q \leq \ep\|xd^{1+\alpha}\|_q$, thus
  $$\| d^\alpha xd- u_{-t}d^\alpha xdu_t\|_q=\| d^\alpha xd- u_{t}d^\alpha xdu_{-t}\|_q
  \leq 2^{1/q}\ep   \|xd^{1+\alpha}\|_q.$$
  Using the first inequality and unitary invariance again
  $$\|d^{1+\alpha} x  -d^\alpha x d\|_q \leq 3^{1/q} \ep \|xd^{1+\alpha}\|_q.$$
  Then by the $q$-triangle and the H\"older inequalities
  $$\| 2 d^{1+\alpha} x\|_q^q \leq \|   d^{1+\alpha} x + d^\alpha xd\|_q^q +
  3\ep^q \|xd^{1+\alpha}\|_q^q\leq \|   d x + xd\|_p^q +
  3\ep^q \|xd^{1+\alpha}\|_q^q .$$
  We finally get $\|   d x + xd\|_p\geq (2^q-3\ep^q)^{1/q} \|xd^{1+\alpha}\|_q= 2^{1-1/q}\|xd^{1+\alpha}\|_q$ assuming $\| xd+dx\|_p\leq \eta \|x d^{1+\alpha}\|_q$.

  Thus for $\eta$ chosen as above, we must have for all $x\in L_r$ and $d\in L_s^+$ with $\|d\|_s=1$:
  $$\|   d x + xd\|_p\geq \min\{\eta,2^{1-1/q}\}.  \|xd^{1+\alpha}\|_q.$$

  The case $1\leq q\leq 2$ can be done in the exact same way replacing the $q$-triangle inequality by the usual triangle inequality.

  Actually when $q>1$, then $p>1$ and we have a stronger inequality
  $$ \|xd\|_p\leq C_p \| d x + xd\|_p.$$ This follows from the
  complete boundedness of triangular projections on $L_p$, see \cite{JunPar}. Thus we get the inequality with constant $C_p\leq
  C_{(1+\alpha)q}$. The constant can even be made independent of
  $\alpha$, we sketch an argument.  As in Corollary 2.3 of \cite{RX}, for any $n$ and any sequence $d_i>0$, the matrix $(\frac
  {d_i^{1+\alpha}+d_j^{1+\alpha}}{(d_i+d_j)(d_i^\alpha+d_j^\alpha)})$
  defines a Schur multiplier on $M_n$  with
  norm less than $\frac 52$. It follows
  $$ \|xd^{1+\alpha}\|_q\leq C_q \| d^{1+\alpha} x + xd^{1+\alpha}\|_q\leq
  \frac 5 2 C_q  \| d^{1+\alpha} x + d^\alpha xd+ d x d^\alpha + xd^{1+\alpha}\|_q.$$
  Thus using
  the H\"older and the triangle inequality, we end up with 
    $$ \|xd^{1+\alpha}\|_q \leq 5C_q \|d\|_s^\alpha \| dx+xd\|_p.$$
\end{proof}

\begin{rk}{\rm
For  integer values of $\alpha$, it is possible to give a simple
algebraic proof using the same kind of ideas but without using the complex uniform convexity.}
\end{rk}

Let us give an easy corollary that improves an estimate from \cite{PR} which is relevant only for $p\leq 1$:
\begin{cor}
  Let $0<\ep<1$, $0<s<\infty$ and $0<r\leq \infty$ let 
  $\frac 1 p = \frac 1s+\frac 1r$. Then there is a constant $C_{\ep,p}$ so that for any $d\in L_s^+$  and $x\in L_r$:
$$\| xd\|_p \leq C_{\ep,p} \big(\|d\|_s \|x\|_r)^\ep .\| dx+xd\|_p^{1-\ep}.$$
\end{cor}  
\begin{proof}
  This is a consequence of  the complex interpolation   \cite[Lemma 2.5]{PR}  for the function $F(z)=xd^{z/(1-\ep)}$ (with an approximation argument as above to make it in $AF(\M)$). Set $q$ so that $\frac 1 q= \frac {1-\ep} p +\frac \ep r$, 
  we  have that $\|xd\|_p\leq  \|xd^{1/(1-\ep)}\|_q^{1-\ep}  \|x\|_r^\eps$. We can conclude applying Theorem \ref{main} with $\alpha =\frac \ep{1-\ep}$.

\end{proof}
  
\subsection{Applications to the Mazur maps}

On a semifinite von Neumann algebra $(\M,\tau)$, the left Mazur map
$M_{p,q} :L_p\to L_q$ is given by $M_{p,q}(f)=f|f|^{p/q-1}$. It is
a uniformly continuous bijection on balls by \cite{Ray}. We determined its
modulus of continuity in our previous papers \cite{R,R2} except in the
case $p>q$ when $q<1$. It still coincides with its commutative analogue:
\begin{thm}\label{maz}
Let $(\M,\tau)$ be a semifinite von Neumann algebra, then for
$q<p<\infty$, the Mazur map $M_{p,q}$ is Lipschitz on balls:
$$\forall x,y\in L_p,\qquad \|M_{p,q}(x)-M_{p,q}(y)\|_q\leq C_{p,q} \max\{\|x\|_p,\|y\|_p\}^{p/q-1} . \| x-y\|_p.$$
  \end{thm}

From the arguments of the papers \cite{R, R2}, one needs to prove that
\begin{equation}\label{eq1}
  \forall x\in \M, \;\forall d\in L_p^+, \qquad \| xd^{p/q}\pm d^{p/q}x\|_q\leq C_{p,q}
  \| xd\pm d x\|_p \|d\|_p^{p/q-1}.
  \end{equation}

The inequality \eqref{eq1} with the plus signs is an easy consequence
of Theorem \ref{main}. The inequality  with the minus sign is also equivalent to
(see \cite[Lemma 2.4]{R})
\begin{equation}\label{eq2}\forall x,y\in L_p^+,  \qquad \| x^{p/q}-y^{p/q}\|_q \leq C'_{p,q}
  \max\{\|x\|_p,\|y\|_p\}^{p/q-1} . \| x-y\|_p.
 \end{equation}  

To prove \eqref{eq2}, one can assume  again that 
both $x$ and $y$ sit in a finite von Neumann algebra and are
invertible with finite spectra. We will need an intermediate result.

We assume that $(\M,\tau)$ is finite and let $d=\sum_{i=1}^n d_i p_i$ with $d_i>0$ and where $p_i$ are orthogonal projections summing up to 1. We define a map $T_{\beta,\gamma}^d$ on $\M$ for $0<  \gamma<1$ and $\beta\geq 0$ by
$$T^d_{\beta,\gamma}(\delta)= \sum_{i,j=1 }^n \frac
{d_i^{\gamma}-d_j^{\gamma}} {d_i-d_j} d_i^\beta d_j^{\beta} p_i\delta
p_j,$$ when $i=j$, $\frac {d_i^{\gamma}-d_j^{\gamma}} {d_i-d_j}$ means
$\gamma d_i^{\gamma-1}$. Note that $T^d_{\beta,\gamma}$ is completely
positive when $0<\gamma<1$ as the matrix $\Big(\frac {d_i^{\gamma}-d_j^{\gamma}} {d_i-d_j}\Big)_{i,j=1}^n$ is positive when $0<\gamma\leq 1$.

We write $\alpha=2\beta+\gamma-1$, and as in Theorem main with $0<s<\infty$ and
$0<r\leq \infty$, we set $p$ so that $\frac 1 p = \frac 1s+\frac 1r$ and $q$ so that $\frac 1 q = \frac {1+\alpha}s+\frac 1r$.

\begin{lemma}\label{majlem} With the above notation, for  $1/2\leq  \gamma<1$ and $\beta>0$ so that $\alpha>0$
and $0<q\leq 1$, there is a constant $C_{\beta,\gamma}$ so that for any 
  $\delta$ in $\M$
$$\|T^d_{\beta,_\gamma}(\delta)\|_{q}\leq C_{\beta,\gamma,q} \| \delta\|_p\|d\|_s^\alpha.$$
\end{lemma}
\begin{proof}
  We start with the case $\gamma=\frac 12$. In this case
  $\frac {d_i^{\gamma}-d_j^{\gamma}} {d_i-d_j}=\frac 1{\sqrt{d_i}+\sqrt{d_j}}$.
  Since we assume $d$ invertible, we can make a change of variable by letting
  $\delta=\sqrt d x+ x\sqrt d$. Then we want to prove that
  $$\|d^\beta xd^\beta\|_q=\|\sum_{i,j=1 }^n d_i^\beta d_j^\beta p_ixp_j\|_q\leq
  C_{\beta,\gamma,q} \|\sqrt d x +x \sqrt d\|_p. \|d\|_s^\alpha.$$ 

  Recall that $\|.\|_q$ is plurisubharmonic (see \cite{Kalton, Xust}) or
  simply by the H\"older inequality $\|d^\beta xd^\beta\|_q\leq
  \max\{\|d^{2\beta}x\|_q ,\|xd^{2\beta}\|_q\}$. Note that
  $2\beta>\frac 12$, we can write $4\beta=1+ 2 \alpha$. We can apply
  Theorem \ref{main} for $x$ and $x^*$ with $\sqrt d$ instead of $d$,
  $2\alpha$ instead of $\alpha$ and $2 s$ instead of $s$ to get
  $$\max\{ \| x{\sqrt d}^{1+2\alpha}\|_q,\| {\sqrt d}^{1+2\alpha}x\|_q\} \leq C_{2\alpha,q} \|\sqrt{d}\|_{2s}^{2\alpha} .\| \sqrt dx+x\sqrt d\|_p,$$
  as  $\frac 1 q=\frac 1 p +\frac {2\alpha}{2s}$. This is exactly the result.

  Let $\frac 12 <\gamma<1$. As we can decompose $\delta$ into the sum of 4 positive elements (with a control on the norm), it suffices to prove the inequality when $\delta\geq 0$.  Put $v=1/(2\gamma)<1$, we use the relation
  $$T^{d^\gamma}_{(1-v)/2,v}T^d_{\beta,\gamma}=T^d_{\beta+\gamma(1-v)/2,
    1/2}=T^d_{\beta+(\gamma -1)/2, 1/2}.$$ The map
  $T^{d^\gamma}_{(1-v)/2,v}$ is completely positive with
  $T^{d^\gamma}_{(1-v)/2,v}(1)=v$.  The Hansen-Pedersen inequality ensures that
  $S(y^q)\leq S(y)^q$ for $y\geq0$, $0<q<1$, and
  $S$ unital completely positive.  We apply it for $y=v
  T^d_{\beta,\gamma}(\delta)$ which is positive and $S=\frac 1v
  T^{d^\gamma}_{(1-v)/2,v}$ which is also trace preserving to get
  $$ \|T^d_{\beta,\gamma}(\delta)\|_q \leq v^{1-q} \| T^d_{\beta+(\gamma-1)/2, 1/2}(\delta)\|_q.$$ 
  We apply the result for $\gamma=1/2$, because $\alpha=2\beta+\gamma-1>0$:
  $$\|T^d_{\beta,\gamma}(\delta)\|_q \leq C_{\beta, \gamma,q} \|\delta\|_p \|d\|_s^\alpha.$$
    
\end{proof}

\begin{rk}{\rm One can get the lemma for $0<\gamma<1$ using compositions but we won't need it.}\end{rk}

\begin{rk}{\rm For $1\leq \gamma\leq 2$, the function $x\mapsto x^\gamma$ is operator convex on $L_p^+$. If $x=y+\delta$ with $\delta\geq0$ assuming $\|\delta\|_p\leq 1=\|x\|_p$, by convexity
    $$ x^\gamma-y^\gamma\leq (1-\|\delta\|_p) y^\gamma + \|\delta\|_p
    (y +\frac 1 {\|\delta\|_p} \delta)^\gamma -y^\gamma$$ Thus, there
    is an easy lipschitz norm control in $L_p$ for the positive part of $x^\gamma-y^\gamma$.
    To get another estimate for the negative part, one can use another operator
    inequality to reduce it to the derivative of the $\gamma$-power function at $y$ on $\delta$. This is the is the content of Lemma \ref{majlem} which can be seen as a control of the defect of positivity of  $x^\gamma-y^\gamma$. 

}\end{rk}
    
\begin{proof}[End of proof of Theorem \ref{maz}]
  We want to prove \eqref{eq2} when $(\M,\tau)$ is finite
  $x=\sum_{i=1}^n x_i p_i$ is positive invertible with a finite spectrum
  and similarly $y=\sum_{j=1}^n  y_j q_j$. 

  We assume for the moment that $p< 3q/2$ so that $p/q=1+t$ with
  $0<t<1/2$. We have
\begin{eqnarray*} x^{p/q}-y^{p/q}&= &\sum_{i=1}^n\sum_{j=1}^m
  \frac {x_i^{1+t}-y_j^{1+t}}{x_i-y_j} p_i(x-y) q_j\\
  &=& x^t (x-y) + (x-y) y^t - \sum_{i=1}^n\sum_{j=1}^m
  \frac {x_i^{1-t}-y_j^{1-t}}{x_i-y_j} x_i^ty_j^t p_i(x-y) q_j
  \end{eqnarray*}
Let $z$ denote the last term of the sum. We use the standard
$2\times2$ trick considering $M_2(\M)$ with trace ${\rm Tr}\otimes
\tau$. We let $d=\left[\begin{array}{cc} x & 0 \\0&
    y\end{array}\right]$ and $\delta= \left[\begin{array}{cc} 0 & x-y
    \\0& 0\end{array}\right]$, by Lemma \ref{majlem} with $\gamma
=1-t$, $\beta=t$ ($\alpha=t>0$) and $s=p$, we get
$$\|z\|_q =\| T^d_{t,1-t} (\delta)\|_q\leq C_{t,q} \|x-y\|_p . (\|x\|_p^p+\|y\|_p^p)^{t/p}.$$

To get the general case, it suffices to notice that $M_{q,r}M_{p,q}= M_{p,r}$
so that one can iterate the result when $0<t<1/2$ to get that for all $t>0$.
\end{proof}

Compared with Lemma \ref{majlem}, it is even more direct to show, with the
notation of the proof of Theorem \ref{maz}, that the maps
$\delta \mapsto \sum_{i=1}^n\sum_{j=1}^m
\frac {x_i^{p/q}-y_j^{p/q}}{x_i+y_j} p_i\delta q_j$ is bounded from
$L_p$ to $L_q$ when $p>q$. By the same arguments as in \cite{R2}, one gets
\begin{cor}
Let $(\M,\tau)$ be a semifinite von Neumann algebra, then for
$q<p<\infty$ :
$$\forall x,y\in L_p,\qquad \| |x|^{p/q}-|y|^{p/q}\|_q\leq C_{p,q} \max\{\|x\|_p,\|y\|_p\}^{p/q-1} . \| x-y\|_p.$$
\end{cor}
\subsection{Extension to type III}

Theorems \ref{main} and \ref{maz} and their corollaries extend in the exact same form to
type III von Neumann algebras using Haagerup's definition. This
follows from the Haagerup reduction principle \cite{HJX} (or
\cite{CPPR} for weights).
The arguments are given in the last section of \cite{R}
with the only difference that one has to avoid the use of conditional
expectations but rather use simple approximations that are provided by
the reduction principle ($\cup_n L_p(\R_n)$ is dense in $L_p(\R)$).  The key point is that this reduction
principle is compatible with the functional calculus and we know from
\cite{Ray} that powers and Mazur maps are continuous.  We leave the
details to the interested reader.

\textbf{Acknowledgement.}  The author is supported by
ANR-19-CE40-0002.

\bibliographystyle{plain}

\end{document}